\documentclass[12pt]{amsart}

\usepackage{amssymb}

\usepackage{enumitem}
\usepackage{tikz}
\usepackage{tikz-cd}
\usetikzlibrary{matrix}
\usepackage{graphicx}
\usepackage{xcolor}
\usepackage{comment}
\makeatletter
\@namedef{subjclassname@2010}{%
  \textup{2020} Mathematics Subject Classification}
\makeatother

\usepackage[T1]{fontenc}

\usepackage{amsmath,amssymb,amsthm,amsfonts,latexsym,amscd,hyperref}
\allowdisplaybreaks
\hypersetup{colorlinks,linkcolor={blue},citecolor={blue},urlcolor={blue}}
\usepackage{xcolor}
\usepackage{mathtools}
\newtheorem{thm}{Theorem}[section]
\newtheorem{cor}[thm]{Corollary}
\newtheorem{lem}[thm]{Lemma}
\newtheorem{prop}[thm]{Proposition}
\newtheorem{qns}[thm]{Question}
\theoremstyle{definition}
\newtheorem{defn}[thm]{Definition}
\theoremstyle{remark}
\newtheorem{rem}[thm]{Remark}
\theoremstyle{note}
\newtheorem{note}[thm]{Note}
\numberwithin{equation}{section}
\newcommand{\vertiii}[1]{{\left\vert\kern-0.25ex\left\vert\kern-0.25ex\left\vert #1
    \right\vert\kern-0.25ex\right\vert\kern-0.25ex\right\vert}}

\begin{document}


\baselineskip=17pt


\title[]{Approximation property in terms of Lipschitz maps via tensor product approach}
\author[Arindam Mandal]{Arindam Mandal$^\dagger$}
\address{Arindam Mandal, School of Mathematical Sciences, National Institute of Science Education and Research
Bhubaneswar, An OCC of Homi Bhabha National Institute, P.O. Jatni, Khurda, Odisha 752050,
India.}
\email{arindam.mandal@niser.ac.in}
\thanks{The author is supported by a research fellowship from the Department of Atomic Energy (DAE), Government of India.}
\thanks{$\dagger$ \tt{Corresponding author}}
\date{}

\begin{abstract}
This article explores the extension of the classical approximation property and its variants to the nonlinear framework of Lipschitz operator theory. Building on Grothendieck’s tensor product methodology, we characterize the Lipschitz approximation property of Banach spaces using Lipschitz finite-rank operators and tensor products. Furthermore, inspired by the $p$-approximation property defined via $p$-compact sets, we introduce and examine the Lipschitz $p$-approximation property. We also establish a factorization theorem for dual Lipschitz $p$-compact operators, mirroring known linear results. This paper looks more closely at how the Lipschitz approximation property and the 
$p$-approximation property of a Banach space are related to those of its Lipschitz-free space. 
\end{abstract}


\maketitle
\pagestyle{myheadings}
\markboth{Arindam Mandal}{Approximation property in terms of Lipschitz maps via tensor product approach}
\section{\bf Introduction and Preliminaries}\label{intro}

The approximation property (a.p) is a cornerstone in the theory of Banach spaces and has been deeply intertwined with the theory of tensor products since the seminal work of Grothendieck \cite{PTTEEN}. Central to his approach is the notion that relatively compact subsets of Banach spaces can be contained within the convex hull of norm-null sequences. Formally, a Banach space 
$X$ is said to have the approximation property if the identity operator 
$Id_X$ lies in the closure (under the topology of uniform convergence on compact sets, 
$\tau_c$) of the space of finite-rank operators 
$\mathcal{F}(X, X)$. Classically, this property admits several equivalent formulations, including:
\begin{itemize}
    \item Uniform approximation of the identity operator on compact subsets by finite-rank operators;
    \item Approximation of compact operators by finite-rank operators in operator norm;
   \item The definition of the trace of nuclear operators via the projective tensor product.
\end{itemize}
In recent years, significant attention has turned toward nonlinear analogues of classical linear theories, especially in the context of Lipschitz mappings. A major development in this direction is the introduction of the Lipschitz approximation property (in short $Lip$-a.p), as formulated by Mingu Jung and Ju Myung Kim in \cite{APITLM}. A Banach space 
$X$ is said to have the $Lip$-a.p if 
$Id_X \in \overline{Lip_{0\mathcal{F}}(X, X)}^{\tau_c}$, where 
$Lip_{0\mathcal{F}}(X, X)$ denotes the space of Lipschitz finite-rank operators. This raises a natural and compelling question:
\begin{qns} \label{FOP} 
Can the Lipschitz approximation property be characterized in terms of tensor product techniques and approximation by Lipschitz compact operators, akin to the classical setting?
\end{qns}
Alongside classical advancements in approximation theory and drawing inspiration from Grothendieck's foundational work, Karn and Sinha introduced the concept of 
$p$-compact sets and the associated 
$p$-approximation property ($p$-a.p, in short) in \cite{COWAFTSLP}. Let us now consider the space of strongly $p$-summable sequences on $X$,  $\ell_p^s(X)$ (forms a Banach space with respect to usual $\|\cdot\|_p$).
A set $K \subset X$ is said to be relatively $p$-compact, for $1 \leq p < \infty$ (with $p'$ being the conjugate exponent of $p$), if there exists a strongly $p$-summable sequence $x = (x_n) \in \ell_p^s(X)$ such that $K$ is contained within the $p$-convex hull, defined as $\left\{\sum\limits_{n=1}^{\infty} a_n x_n : (a_n) \in B_{\ell_{p'}} \right\}$. In this setup, $p$-compact operators are those that map bounded sets to relatively $p$-compact sets. A Banach space 
$X$ has the 
$p$-approximation property if, for every relatively 
$p$-compact set 
$K \subset X$ and 
$\epsilon > 0$, the identity operator can be uniformly approximated on 
$K$ by a finite-rank operator. Formally:
\begin{defn}[\cite{COWAFTSLP}, Definition 6.1] 
A Banach space 
$X$ has the 
$p$-approximation property if for every 
$p$-compact set 
$K \subset X$ and every 
$\epsilon>0$, there exists a finite-rank operator 
$T$ such that 
$\|Tx-x\| < \epsilon$ for all 
$x \in K$.
\end{defn}
Naturally, one is led to ask whether this notion can be extended to the nonlinear, Lipschitz setting:

\begin{qns} \label{qns2} 
Can the classical 
$p$-approximation property be generalized to the context of Lipschitz operators? 
\end{qns}

In this article, we answer both Questions \ref{FOP} and \ref{qns2} affirmatively. We develop a characterization of the Lipschitz approximation property through the lens of tensor products and Lipschitz compact operators, and we define and study the Lipschitz 
$p$-approximation property as a natural nonlinear extension of the classical theory. Furthermore, we provide a factorization result for Lipschitz operators whose transposes lie in the dual of the Lipschitz 
$p$-compact operator ideal, mirroring the linear result of Karn and Sinha \cite{COWFTSLP}.

To support our results, we begin by reviewing the foundational concepts from the theory of Lipschitz operator ideals, Lipschitz-free spaces, and tensor products. Even though several results continue to hold in the broader setting of normed linear spaces, we focus on real Banach spaces 
$X$ and 
$Y$ throughout the paper to prevent confusion, unless specified. $B_X$ will denote the closed unit ball of $X$ where whereas $X \cong Y$ stands for $X$ is linear isometrically isomorphic to $Y$. Furthermore, for any Banach space $Y$ its dual will be denoted by $Y^{*}$. A map $f : X \xrightarrow{} Y$ is said to be Lipschitz if there exists $c\geq 0$ such that $\|f(x_{1})-f(x_{2})\| \leq c\|x_{1}-x_{2}\| \ $ for any $x_{1}, x_{2} \in X$. The space $Lip_{0}(X, Y) = \{ f: X \xrightarrow{}Y \ |\ f \mbox{ is Lipschitz and \ } f(0)=0\}$ is a Banach space concerning the Lipschitz norm, defined as $$Lip(f) = \sup\limits_{x_1, x_2 \in X;\ x_1\neq x_2} \frac{\|f(x_1)-f(x_2)\|}{\|x_1-x_2\|},$$ whenever $Y$ be a Banach space.  Elements within $Lip_{0}(X,Y)$ are commonly referred to as Lipschitz operators. In the case where $Y=\mathbb{R}$, we denote $Lip_{0}(X,\mathbb{R})$ by $X^{\#}$. 
For any $x \in X$ let us define the map  $\delta_{x} : X^{\#} \xrightarrow{} \mathbb{K}$ such that $\delta_{x}(f)=f(x)$. Then $\delta_{x}$ is bounded linear. Now the Lipschitz free space $F(X)$ is defined as $F(X) = \overline{span\{\delta_{x} : x \in X\}}^{\|.\|} \ \subset (X^{\#})^{*}$. It has been proven that the dual of $F(X)$ is $X^{\#}$ by several authors for example Nik Weaver in \cite{LA}, Vargas, Sepulcre et al. in \cite{LCO}.
The spaces $X$ and $F(X)$ are related by the Lipschitz map $\delta_{X}: X \xrightarrow{} F(X)$ such that $\delta_{X}(x) = \delta_{x}$ with $Lip(\delta_{X}) = 1$. By $$Lip_{0 \mathcal{K}}(X, Y)= \left\{f \in Lip_{0}(X, Y) : \left\{ \frac{f(x_{1})-f(x_{2})}{\|x_{1}-x_{2}\|} : x_{1}, x_{2} \in X, x_{1} \neq x_{2}\right\} \mbox{is relatively compact}  \right\}$$ we denote the set of Lipschitz compact operators. We denote the set of all bounded linear operators between the normed linear spaces $X$ and $Y$ by $L(X, Y)$, a subspace of  $Lip_{0}(X, Y)$. Additionally, we denote the complete projective tensor product of $X$ and $Y$ by $X \hat{\otimes}_{\pi} Y$ which consists of all the elements $\mathfrak{U}$ having a representation of the form $\mathfrak{U} = \sum\limits_{n=1}^{\infty} e_{n} \otimes f_{n}$, where $(e_{n})$ and $(f_{n})$ are bounded sequences in $X$ and $Y$, respectively, such that $\sum\limits_{n=1}^{\infty} \|e_{n}\|\|f_{n}\| < \infty$. 
Specifically, we explore the structure of the Banach space 
$Lip_0(X, Y)$, the role of the Lipschitz-free space 
$F(X)$, and the projective tensor product 
$X \hat{\otimes}_{\pi} Y$, all of which play central roles in our arguments. For background on tensor products of Banach spaces and their completions under different norms, we refer to \cite{ITTPOBS}.

We also draw upon and extend earlier studies into operator ideals, particularly the work of Achour, Dahia, and Turco \cite{LPCM} on Lipschitz 
$p$-compact operators and their ideals, as well as broader studies of Lipschitz operator theory by Achour, Rueda et al. \cite{LOITAP}. For background on classical operator ideals, we refer the reader to \cite[Chapters 2, 5]{ASO}.

This paper is organized as follows:
In Section \ref{TPVOLAPBS}, we present a tensor product characterization of the Lipschitz approximation property and establish that the approximation property of the Lipschitz dual 
$X^{\#}$
  implies Lipschitz approximation property for 
$X$. Section \ref{LPAPOBS} is devoted to defining and analyzing the Lipschitz 
$p$-approximation property, including examples of Banach spaces and their free spaces where this property holds.
In Section \ref{FOLPCO}, we prove a factorization theorem for certain Lipschitz operators in terms of 
$p$-summing and Lipschitz compact operators.
Finally, we conclude with remarks on the broader implications and potential future directions in the study of approximation properties within both linear and nonlinear frameworks.

Before proceeding further, we introduce the notations that will be used consistently throughout this article for the sake of clarity. The space of all 
$p$-compact operators from a Banach space 
$
X$ to another Banach space 
$Y$ is denoted by 
$\mathcal{K}_{p}(X, Y)$, while its dual operator ideal space is represented by 
$\mathcal{K}_{p}^{d}(X, Y)$, defined as $\mathcal{K}_{p}^{d}(X, Y):= \left\{T \in L(X, Y): T^* \in \mathcal{K}_{p}(Y^*, X^*)\right\}$ (see \cite[Chapter 4]{OI} and \cite{COWAFTSLP}). The class of 
$p$-summing operators between these spaces is written as 
$\prod_{p}(X, Y)$, with the respective norms denoted by 
$k_{p}(.)$, $k_{p}^{d}(.)$ and $\pi_{p}(.)$.

Within the setting of Lipschitz operators we denote Lipschitz 
$p$-compact operators, their duals, and Lipschitz 
$p$-summing operators by 
$Lip_{0K}(X, Y)$, $\mathcal{K}_{p}^{L}(X, Y)$, $\mathcal{K}_{p}^{Ld}(X, Y)$
and $\prod_{p}^{L}(X, Y)$, respectively. The associated norms are denoted as 
$k_{p}^{L}(.)$, $k_{p}^{Ld}(.)$, $\pi_{p}^{L}(.)$, depending on the context.

For more detailed discussions on these operator ideals and Lipschitz classes, we refer the reader to \cite{LOITAP}, \cite{ASO}, \cite{OI}, \cite{LCO}, and \cite{LPCM}. Throughout the rest of this article, we will also utilize key structural properties of Lipschitz-free Banach spaces, as developed in works such as \cite{LCO} and \cite{SOLAHFATA}, without further citation.
\begin{lem} \label{Lem}
\begin{enumerate}
    \item The space $X^{\#}$ can be mapped isometrically onto $F(X)^{*}$ via the transformation $Q_{X} : X^{\#} \xrightarrow{} F(X)^{*}$ given by $Q_{X}(f)(\gamma)=\gamma(f);$ $\ \mbox{for all}\ f \in X^{\#}, \gamma \in F(X)$.
    \item  For any Lipschitz mapping $f \in Lip_{0}(X,Y) \ \mbox{there exists}$ a unique linear map $\hat{f}: F(X) \xrightarrow{} F(Y)$ such that $\hat{f} \circ \delta_{X} = \delta_{Y} \circ f$ with $\|\hat{f}\| = Lip(f).$
    \item There exists a quotient map $\beta_{X} : F(X) \xrightarrow{} X$ such that $e^{*}(\beta_{X}(\mu)) = \mu(e^{*})$ for all $\mu \in F(X)$ and $e^{*} \in X^{*}$ such that $\beta_{X} \circ \delta_{X} = Id_{X}$.
    \item For every Lipschitz operator $f \in Lip_{0}(X,Y)$ there exists a unique bounded linear map $T_{f} : F(X) \xrightarrow{} Y$ such that $T_{f} \circ \delta_{X} = f$ with $\|T_{f}\|= Lip(f).$ We call this map $T_{f}$ the linearization of $f$. Also, $f \in Lip_{0 \mathcal{F}}(X,Y)$ if and only if $T_{f} \in \mathcal{F}(F(X), Y)$.
\end{enumerate}
\end{lem}
Let us denote by $\tau_{c},\tau_{p}$ the topologies of uniform convergence on compact and $p$-compact subsets of $X$, respectively, and also $\mathcal{T}\delta_{c},\mathcal{T}\delta_{p}$ the locally convex topologies on the space $L(F(X),Y)$ generated by the seminorms 
$$P_{K}(T)=\sup_{x \in K} \|T(\delta_{X}(x))\|$$ for every compact and $p$-compact subset $K$ of $ X$, respectively; for all $T \in L(F(X),Y)$.
\begin{rem}
    As any $p$-compact set is compact $\tau_{p}$ is weaker than $\tau_{c}$.
\end{rem}
\section{\bf Tensor product version of Lipschitz approximation property of a Banach space} \label{TPVOLAPBS}
As the central outcome of this section, we unveil the subsequent characterization of Lipschitz approximation property in terms of the tensor product in Theorem \ref{thm3.2}, and the following lemma will be used in proving it. 
\begin{lem} \label{lem1.1}
    Let $X$ be a Banach space such that $ Id_{F(X)} \in \overline{\mathcal{F}(F(X),F(X))}^{\mathcal{T}\delta_{c}}$ then for all Banach space $Y$, $L(F(X),Y)\subset\overline{\mathcal{F}(F(X),Y)}^{\mathcal{T}\delta_{c}}$.
\end{lem}
\begin{proof}
    Let $Y$ be a Banach space, $T (\neq 0) \ \in L(F(X),Y)$, $K$ be a compact subset of $X$ and $\epsilon>0$.
    Therefore, there exists $S \in \mathcal{F}(F(X),F(X))$ such that $\sup\limits_{x \in K}\|\delta_{X}(x) - S\delta_{X}(x)\| < \frac{\epsilon}{\|T\|}$.\\
    Now
    \begin{eqnarray*}
        \sup\limits_{x \in K}\|T\delta_{X}(x) - TS\delta_{X}(x)\| \leq \|T\| \sup\limits_{x \in K}\|\delta_{X}(x) - S\delta_{X}(x)\| < \epsilon 
    \end{eqnarray*}
    Hence, the result follows.
\end{proof}
The following theorem provides a tensor-product characterization of the Lipschitz approximation property, thereby answering Question \ref{FOP}.
\begin{thm}\label{thm3.2}
    For a Banach space $X$, the following are equivalent:
    \begin{enumerate}
        \item $X$ has $Lip- a.p$
        \item  $ Id_{F(X)} \in \overline{\mathcal{F}(F(X),F(X))}^{\mathcal{T}\delta_{c}}$.
        \item For any Banach space $Y$ and if $\mathfrak{U} = \sum\limits_{n=1}^{\infty} \delta_{x_{n}} \otimes y_{n} \ \in F(X) \hat{\otimes}_{\pi} Y$ where $(x_{n})$ and $(y_{n})$ are bounded sequences in $X$ and $Y$, respectively, such that $\sum\limits_{n=1}^{\infty} \|x_{n}\|\|y_{n}\| < \infty$ and if $\sum\limits_{n=1}^{\infty}\psi(y_{n}) \delta_{x_{n}} = 0, \ \mbox{for all} \ \psi \in Y^{*}$, then $\mathfrak{U} = 0.$
        
        \item For any Banach space $Y$ and if $\mathfrak{U} = \sum\limits_{n=1}^{\infty} \delta_{x_{n}} \otimes y_{n} \ \in F(X) \hat{\otimes}_{\pi} Y$ where $(x_{n})$ and $(y_{n})$ are bounded sequences in $X$ and $Y$ respectively such that $\sum\limits_{n=1}^{\infty} \|x_{n}\|\|y_{n}\| < \infty$ and if $\sum\limits_{n=1}^{\infty} \phi(\delta_{x_{n}})y_{n}  = 0, \ \mbox{for all} \ \phi \in F(X)^{*}$, then $\mathfrak{U} = 0.$ 

        \item If $\mathfrak{U} = \sum\limits_{n=1}^{\infty} \phi_{n} \otimes \delta_{x_{n}} \ \in F(X)^{*} \hat{\otimes}_{\pi} F(X)$ where $(\phi_{n})$ and $(x_{n})$ are bounded sequences in $F(X)^{*}$ and $X$ respectively such that $\sum\limits_{n=1}^{\infty} \|\phi_{n}\|\|x_{n}\| < \infty$ and if $\sum\limits_{n=1}^{\infty}\phi_{n}(\delta_{x}) \delta_{x_{n}} = 0, \ \mbox{for all} \ x \in X$, then $\mathfrak{U} = 0.$ 
        
    \end{enumerate}
\end{thm}
\begin{proof}
    $(1) \iff (2)$ part has already been proved in \cite[Theorem~3.1]{APITLM}. 
    
    $\underline{(2) \implies(3)}$ \\
    Let $\mathfrak{U}$ have such a representation for a Banach space $Y$ as in $(3)$ with $\sum\limits_{n=1}^{\infty}\psi(y_{n}) \delta_{x_{n}} = 0, \ \mbox{\ for \ all} \ \psi \in Y^{*}$. We shall show $T(\mathfrak{U})=0$ for all $T \in (F(X) \hat{\otimes}_{\pi} Y)^{*} \cong L(F(X),Y^{*})$.
    
    Since $\sum\limits_{n=1}^{\infty} \|x_{n}\|\|y_{n}\| < \infty$ without loss of generality, we can assume $(x_{n})$ converges to $0$ and $\sum\limits_{n=1}^{\infty}\|y_{n}\| < \infty$.
    Suppose $T \in L(F(X),Y^{*})$ and $\epsilon>0.$ Therefore, by Lemma \ref{lem1.1}, for the compact set $\{x_{n} : n \in \mathbb{N}\} \cup \{0\}$ $\mbox{\ there  exists  } \ S \in \mathcal{F}(F(X),Y^{*})$ such that $\|T\delta_{X}(x_{n}) - S\delta_{X}(x_{n})\| < \epsilon$ for all $n \in \mathbb{N}$. Since $S$ is a finite-rank operator in $L(F(X),Y^{*})$, $S = \sum\limits_{i=1}^{m} \phi_{i} \otimes \psi_{i}$, where $\phi_{i}\in F(X)^{*}$ and $\psi_{i}\in Y^{*}$ for $i =1,...,m$. It is easy to verify that $S(\mathfrak{U})=0$.\\
    Now, 
    \begin{eqnarray*}
        |T(\mathfrak{U})| = |(T-S)(\mathfrak{U}|&=&|\sum\limits_{n=1}^{\infty} (T-S)(\delta_{x_{n}} \otimes y_{n} )|\\
        &\leq& \sum\limits_{n=1}^{\infty} \|(T-S)(\delta_{x_{n}})\|\|y_{n}\|\\
        &\leq& \epsilon \sum\limits_{n=1}^{\infty}\|y_{n}\|. 
    \end{eqnarray*}
    Since $\epsilon>0$ is arbitrary $\mathfrak{U} = 0$.\\
  $\underline{(3) \implies(4)}$\\
  Let us assume $Y$ is a Banach space and $\mathfrak{U} = \sum\limits_{n=1}^{\infty} \delta_{x_{n}} \otimes y_{n} \ \in F(X) \hat{\otimes}_{\pi} Y$, where $(x_{n})$ and $(y_{n})$ are bounded sequences in $X$ and $Y$, respectively, with $\sum\limits_{n=1}^{\infty} \phi(\delta_{x_{n}})y_{n}  = 0 \ \mbox{for all} \ \phi \in F(X)^{*}$. Now, for all $ \phi \in F(X)^{*}, \psi \in Y^{*}$; $0 = \psi(\sum\limits_{n=1}^{\infty} \phi(\delta_{x_{n}})y_{n}) = \phi(\sum\limits_{n=1}^{\infty} \psi(y_{n})\delta_{x_{n}})$. Hence, it immediately follows $\mathfrak{U} =0$.
  
  $\underline{(4) \implies(5)}$\\
  Suppose $\mathfrak{U} = \sum\limits_{n=1}^{\infty} \phi_{n} \otimes \delta_{x_{n}} \ \in F(X)^{*} \hat{\otimes}_{\pi} F(X)$, where $(\phi_{n})$ and $(x_{n})$ are bounded sequences in $F(X)^{*}$ and $X$, respectively, such that $\sum\limits_{n=1}^{\infty} \|\phi_{n}\|\|x_{n}\| < \infty$ with $\sum\limits_{n=1}^{\infty}\phi_{n}(\delta_{x}) \delta_{x_{n}} = 0 $ for all $ x \in X$. Therefore, for all $x \in X$ and for any linear functional $\gamma$ on $F(X)$,  $\gamma\big(\sum\limits_{n=1}^{\infty}\phi_{n}(\delta_{x}) \delta_{x_{n}} \big)=0.$ This is equivalent as $\Big(\sum\limits_{n=1}^{\infty} \gamma(\delta_{x_{n}}) \phi_{n}\Big)(\delta_{x}) = 0$ for all $x \in X$ and $\gamma \in F(X)^{*}.$ Thus, for all linear functionals $\gamma$ of $F(X)$, $\sum\limits_{n=1}^{\infty} \gamma(\delta_{x_{n}}) \phi_{n} =0$.  It follows that $\sum\limits_{n=1}^{\infty} \phi_{n} \otimes \delta_{x_{n}} =0$, which is  same as $\mathfrak{U}=0$.
  
$\underline{(5) \implies(2)}$\\
 We will demonstrate that if the statement in $(2)$ is false, the statement in $(5)$ is also false.\\
   Suppose  $ Id_{F(X)} \not \in \overline{\mathcal{F}(F(X),F(X))}^{\mathcal{T}\delta_{c}}$. Consider $E = (L(F(X),F(X)), \mathcal{T}\delta_{c})$ and $G$ be a subspace of $E$ consisting of all finite-rank operators. Therefore, $\{Id_{F(X)}\} \cap \overline{G}^{\mathcal{T}\delta_{c}} = \emptyset$. Hence, by the geometric version of the Hahn-Banach theorem, there exists $\Phi \in E^{*}$ such that $\Phi(T)=0$ for all $T \in G$ and $\Phi(Id_{F(X)}) = 1$.\\
   As $\Phi \in E^{*}$, there exists a compact subset $K$ of $X$ such that $|\Phi(T)| \leq P_{K}(T) \ \ \forall \ T \in E$. Since $K$ is compact, there exists $(x_{n}) \in c_{0}(X)$ such that $K$ is contained in the closed convex hull of the sequence $(x_n)$. Thus $|\Phi(T)| \leq \sup\limits_{n} \|T\delta_{x_{n}}\|$. 
   
   Consider $Z = \{(T\delta_{x_{n}}) : T \in E\}$ subspace of $c_{0}(F(X))$. Then $\Tilde{\Phi} : Z \xrightarrow{} \mathbb{K}$ is defined as $\Tilde{\Phi}((T\delta_{x_{n}})) = \Phi(T)$ is a well-defined bounded linear map, and let  $\Tilde{\Phi}_{ext}\in \ c_{0}(F(X))^{*}$ be the Hahn-Banach extension. Therefore, $\Tilde{\Phi}_{ext}((T\delta_{x_{n}}))= \Phi(T) = \sum\limits_{n=1}^{\infty}\phi_{n}(T\delta_{x_{n}})$ for $(\phi_{n}) \in \ell_{1}(F(X)^{*})$ and for all $T \in E$. Hence, $1= \Phi(Id_{F(X)}) = \sum\limits_{n=1}^{\infty}\phi_{n}(\delta_{x_{n}})$ while $\Phi(T)=0$ for all $T \in G$, in particular for $T = \Psi \otimes \delta_{x}$ for all $\Psi \in F(X)^{\ast}$ and $x \in X$. 
   This implies for all $\Psi \in F(X)^{\ast}$ and $x \in X$
   $$0 =  \Phi(T) = \sum\limits_{n=1}^{\infty}\phi_{n}(T\delta_{x_{n}}) = \sum\limits_{n=1}^{\infty} \Psi(\delta_{x_{n}}) \Phi_{n}(\delta_{x})= \Psi \left(\sum\limits_{n=1}^{\infty}\Phi_{n}(\delta_{x})\delta_{x_{n}}\right)$$
   Therefore, it follows
   $\sum\limits_{n=1}^{\infty}\phi_{n}(\delta_{x})\delta_{x_{n}} = 0$ for all $x  \in X$. Therefore, $(5)$ does not hold, as if it holds then $\sum\limits_{n=1}^{\infty}\phi_{n}\otimes \delta_{x_{n}} = 0$ and it gives us $\sum\limits_{n=1}^{\infty}\phi_{n}(\delta_{x_{n}})=0$, which is a contradiction. Hence, the proof follows.
\end{proof}
\begin{cor}
    Let $X$ and $Y$ be Banach spaces, with $X$ having $Lip-a.p$. Then the map \\
    $\mathcal{J} : F(X)^{*} \hat{\otimes}_{\pi} Y \xrightarrow{} N(F(X),Y)$ defined as $\mathcal{J}(u)(\delta_{x}) = \sum\limits_{n=1}^{\infty}\phi_{n}(\delta_{x})y_{n}$ is one-one, where $u= \sum\limits_{n=1}^{\infty}\phi_{n} \otimes y_{n}\in F(X)^{*} \hat{\otimes}_{\pi} Y$.
\end{cor}
\begin{proof}
    Let $X$ be a Banach space having $Lip-a.p$ and $u \in Ker(\mathcal{J}) \subset F(X)^{*} \hat{\otimes}_{\pi} Y $. Therefore, there exist bounded sequences $(\phi_{n})$ and $(y_{n})$ in $F(X)^{*}$ and $Y$ respectively with $\sum\limits_{n=1}^{\infty} \|\phi_{n}\|\|x_{n}\| < \infty$ such that $u= \sum\limits_{n=1}^{\infty}\phi_{n} \otimes y_{n}\in F(X)^{*} \hat{\otimes}_{\pi} Y$. Now $\mathcal{J}(u)=0$ gives $\sum\limits_{n=1}^{\infty}\phi_{n}(\delta_{x})y_{n}=0$ for all $x \in X.$ Hence from the above theorem we have $u=0$. This completes the proof.
\end{proof}
\begin{prop} \label{Prop3.4} Let $X$ be a Banach space. Then the following are equivalent:
\begin{enumerate}
    \item $X$ has the Lipschitz approximation property.
    \item For any Banach space $Y$ and any $f \in Lip_{0}(X,F(Y))$,   $ \ \beta_{Y}\circ f \in \overline{Lip_{0\mathcal{F}}(X,Y)}^{\tau_{c}}$.  
    \end{enumerate}
\end{prop}
\begin{proof}
    $(2) \implies (1)$ directly follows by considering $Y=X$ and $f=\delta_{X}$.\\
    \underline{$(1) \implies (2)$}\\ 
    Let $Y$ be a Banach space, $K$ be a compact subset of $X$, $\epsilon > 0$ and $f \in Lip_{0}(X,F(Y))$. Therefore by the \cite[Proposition 3.2]{APITLM} there exists $g \in Lip_{0 \mathcal{F}}(X,F(Y))$ such that $\sup\limits_{x \in K}\|f(x) - g(x)\| < \epsilon $.\\
    Now $\sup\limits_{x \in K}\|\beta_{Y} \circ f (x) - \beta_{Y} \circ g (x)\| \leq \sup\limits_{x \in K}\|f(x) - g(x)\| < \epsilon$ and it immediately proves this part.
\end{proof}
The following proposition will link the approximation property of $F(X)$ with the Lipschitz approximation property of the Banach space $X$.
\begin{prop}
    If the Lipschitz-free space $F(X)$ of a Banach space $X$ has approximation property, then $X$ has the Lipschitz approximation property.
\end{prop}
\begin{proof}
     Suppose $X$ is a Banach space such that $F(X)$ has $a.p$, $K$ is a compact subset of $X$, and $\epsilon >0$. Therefore, there exists $ S \in \mathcal{F}(F(X),F(X))$ such that $ \sup\limits_{x \in K}\|Id_{F(X)}(\delta_{X}(x)) - S\delta_{X}(x)\| < \epsilon$.\\
     Now, $ \sup\limits_{x \in K}\|Id_{X}(x) - \beta_{X}S\delta_{X}(x)\| \leq  \sup\limits_{x \in K}\|Id_{F(X)}(\delta_{X}(x)) - S\delta_{X}(x)\| < \epsilon$ \\
     This implies $Id_{X} \in \overline{Lip_{0\mathcal{F}}(X,X)}^{\tau_{c}}$ and hence $X$ has $Lip-a.p$.
\end{proof}
\begin{cor}\label{cor3.5}
    If for any Banach space $X$, $X^{\#}$ has $a.p$ then $X$ has $Lip-a.p$.
\end{cor}
\begin{proof}
Let $X$ be a Banach space such that $X^{\#}$ has $a.p$. Therefore, $F(X)^{*}$ has $a.p$, and this gives that $F(X)$ has $a.p$, as $F(X)$ is a Banach space, and hence $X$ has $Lip-a.p$ (by the previous proposition).
\end{proof}
\begin{prop} \label{prop2.7}
 The following implications hold:
\begin{center}
\begin{tikzcd}
X \ has\ a.p \arrow[r,Rightarrow] \arrow[d,Rightarrow] & \beta_{X} \in \overline{\mathcal{F}(F(X),X)}^{\tau_{c}} \arrow[d,Rightarrow] \\
X \ has \ Lip-a.p \arrow[r,Leftrightarrow] \arrow[d,Leftrightarrow] & \beta_{X} \in \overline{\mathcal{F}(F(X),X)}^{\tau \delta_{c}}\\
\beta_{X} \in \overline{Lip_{0\mathcal{F}}(F(X),X)}^{\tau_{c}}
\end{tikzcd} 
\end{center}
\end{prop} 
\begin{proof}
    From the definitions $X \ has\ a.p \implies X \ has \ Lip-a.p$ part follows.\\
    \underline{Proof of $X \ has\ a.p \implies \beta_{X} \in \overline{\mathcal{F}(F(X),X)}^{\tau_{c}}$}:\\
    Let $A$ be a compact set in $F(X)$ and $\epsilon > 0$, therefore $\beta_{X}(A)$ is compact subset of $X$. Since $X \ has\ a.p$ there exists $S \in F(X,X)$ such that $\sup\limits_{\gamma \in A}\|Id_{X}(\beta_{X}(\gamma)) - S\beta_{X}(\gamma)\| < \epsilon$. This implies that $\beta_{X} \in \overline{\mathcal{F}(F(X),X)}^{\tau_{c}}.$\\
    \\
    \underline{To prove : $X \ has \ Lip-a.p \Longleftrightarrow \beta_{X} \in \overline{\mathcal{F}(F(X),X)}^{\tau \delta_{c}}$}:\\
    Suppose $X \ has \ Lip-a.p$, $K$ be a compact subset of $X$, and $\epsilon >0$. Then there exists $S \in Lip_{o \mathcal{F}}(X, X)$ such that $\sup\limits_{x \in K}\|Id_{X}(x)- S(x)\| < \epsilon$. Now $Id_{X}-S = (\beta_{X} - T_{S}) \circ \delta_{X}$, where $T_{s}$ is the linearization of $S$. Hence $\sup\limits_{x \in K}\|\beta_{X}\big(\delta_{X}(x)\big)- T_{S}\big(\delta_{X}(x)\big)\| =\sup\limits_{x \in K}\|Id_{X}(x)- S(x)\| < \epsilon$ and this gives us the forward implication part.\\
    For the converse let us consider $\beta_{X} \in \overline{\mathcal{F}(F(X),X)}^{\tau \delta_{c}}$, $K$ be a compact subset of $X$, and $\epsilon >0$. Then there exists $S \in \mathcal{F}(F(X), X)$ such that $\sup\limits_{x \in K}\|\beta_{X}\big(\delta_{X}(x)\big)- S\big(\delta_{X}(x)\big)\| < \epsilon$. This immediately provides $X \ has \ Lip-a.p$.\\
    \\
    \underline{To prove : $X \ has \ Lip-a.p \Longleftrightarrow \beta_{X} \in \overline{Lip_{0\mathcal{F}}(F(X),X)}^{\tau_{c}}$}:\\
    Let $A$ be a compact set in $F(X)$ and $\epsilon > 0$, therefore $\beta_{X}(A)$ is compact subset of $X$. Since $X \ has\ Lip-a.p$ there exists $S \in Lip_{0\mathcal{F}}(X,X)$ such that $\sup\limits_{\gamma \in A}\|Id_{X}(\beta_{X}(\gamma)) - S\beta_{X}(\gamma)\| < \epsilon$. This implies that $\beta_{X} \in \overline{Lip_{0\mathcal{F}}(F(X),X)}^{\tau_{c}},$ as $S \circ \beta_{X} \in Lip_{0\mathcal{F}}(F(X),X).$\\
    Conversely, let $\beta_{X} \in \overline{Lip_{0\mathcal{F}}(F(X),X)}^{\tau_{c}}$, $K$ be a compact subset of $X$, and $\epsilon >0$. Then there exists $S \in Lip_{0\mathcal{F}}(F(X),X)$ such that $\sup\limits_{x \in K}\|Id_{X}(x)- S\circ \delta_{X}(x)\| = \sup\limits_{x \in K}\|\beta_{X}\big(\delta_{X}(x)\big)- S\big(\delta_{X}(x)\big)\| < \epsilon$. Hence $X \ has \ Lip-a.p$.\\ Comparing the two topologies $\tau_{\delta_{c}}$ and $\tau_{c}$ we get $\beta_{X} \in \overline{\mathcal{F}(F(X),X)}^{\tau_{c}} \implies \beta_{X} \in \overline{\mathcal{F}(F(X),X)}^{\tau_{\delta_{c}}}$. This completes the proof.
\end{proof}
\begin{rem}
    In a nutshell, from this section, we get the following:
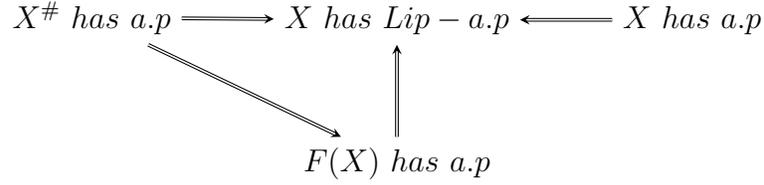
\begin{center}
\begin{tikzpicture}
\matrix (m) 
[matrix of math nodes,row sep=3em,column sep=3em,minimum width=1em] 
{
X^{\#}\ has \ a.p & X \ has \ Lip-a.p & X \ has \ a.p\\
          & F(X) \ has \ a.p \\ 
}; 
\path[-stealth] 
(m-1-1) edge [double] node [above] {} (m-1-2)
        edge [double] node [below] {} (m-2-2)
(m-2-2) edge [double] node [right] {} (m-1-2)
(m-1-3) edge [double] node [left] {} (m-1-2);
\end{tikzpicture}
\end{center}
\end{rem}
Next, we see some more equivalent characterization of the Lipschitz approximation
property of the Banach space $X$.
\begin{prop}
    The following are equivalent:
    \begin{enumerate}
        \item $Id_{X} \in \overline{Lip_{0\mathcal{F}}(X,X)}^{\tau_{c}}$.
        \item $\delta_{X} \in \overline{Lip_{0\mathcal{F}}(X,F(X))}^{\tau_{c}}$.
        \item $\delta_{X}\beta_{X} \in \overline{Lip_{0\mathcal{F}}(F(X),F(X))}^{\tau_{c}}$
    \end{enumerate}
\end{prop}
\begin{proof}
    Let $(1)$ holds, $K$ be a compact subset of $X$ and $\epsilon > 0$. Therefore by theorem \ref{thm3.2} $ Id_{F(X)} \in \overline{\mathcal{F}(F(X),F(X))}^{\mathcal{T}\delta_{c}}$ and there exists $S \in \mathcal{F}(F(X),F(X))$ such that $\sup\limits_{x \in K}\|Id_{F(X)}(\delta_{X}(x))- S\delta_{X}(x)\| < \epsilon$. Since $S \circ \delta_{X} \in Lip_{0\mathcal{F}}(X,F(X))$ this gives $\delta_{X} \in \overline{Lip_{0\mathcal{F}}(X,F(X))}^{\tau_{c}}$. Now suppose $\delta_{X} \in \overline{Lip_{0\mathcal{F}}(X,F(X))}^{\tau_{c}}$, $K$ be a compact subset of $X$ and $\epsilon > 0$. Then there exists $S \in Lip_{0\mathcal{F}}(X,F(X))$ such that $\sup\limits_{x \in K}\|\delta_{X}(x))- S\delta_{X}(x)\| < \epsilon$. Now $\sup\limits_{x \in K}\|Id_{X}(x)- \beta_{X}S(x)\|\leq \sup\limits_{x \in K}\|\delta_{X}(x))- S\delta_{X}(x)\| < \epsilon$. As $\beta_{X} \circ S \in Lip_{0\mathcal{F}}(X,X)$, $(1)$ follows.\\
    Equivalency of $(1)$ and $(3)$ follows directly from $(2)$ and the Proposition \ref{prop2.7}.
\end{proof}
\section{\bf Lipschitz p-approximation property of a Banach space} \label{LPAPOBS}
The definitions of special types of approximation
properties that were studied later (e.g., $p$-approximation properties studied
by Saphar~\cite{HDAPDLEDBEADPA} for $1 \leq p < \infty$ and by Re\v{n}ov~\cite{DOTEITSOPNO} for $0 < p < \infty$)
were always conceived via the tensor product route. Later on, in \cite{COWAFTSLP}, Karn and Sinha examine the approximation of the identity operator
on $p$-compact sets by finite-rank operators, namely the $p$-approximation
property. Motivated by their definition, we formulated Question \ref{qns2} earlier,
In this section, we introduce the Lipschitz $p$-approximation property,
provide several characterizations of it, and establish a connection between
the linear $p$-approximation property of the Lipschitz free space $F(X)$
and the Lipschitz $p$-approximation property of $X$.

\begin{defn}
    A Banach space $X$ is said to have Lipschitz $p-$approximation property in short $Lip-p.a.p$ if $Id_{X} \in \overline{Lip_{0\mathcal{F}}(X,X)}^{\tau_{p}}$. 
\end{defn}
\begin{note}
    From the definition, it is clear that if a Banach space $X$ has linear $p.a.p$, then it has $Lip-p.a.p$, and by comparing the topologies, we have that if $X$ has $Lip-a.p$, then $ X$ has $Lip-p.a.p$.
\end{note}
The first two theorems in this section characterize the Lipschitz $p$-approximation property of $X$. The proof of the theorem is similar to that of \cite[Theorem~3.1]{APITLM} with the following modifications.
\begin{thm}\label{thm0.2}
    A Banach space $X$ has the Lipschitz $p$-approximation property iff $Id_{F(X)} \in \overline{\mathcal{F}(F(X),F(X))}^{\mathcal{T}\delta_{p}}$.
\end{thm}
\begin{proof}
    Let $X$ has Lipschitz $p$-approximation property, $K$ be a $p$-compact subset of $X$ and $\epsilon > 0$. Then there exists $f \in Lip_{0\mathcal{F}}(X,X)$ such that $$\sup\limits_{x \in K} \left\|f(x)-x\right\|\leq \frac{\epsilon}{2}.$$
    Also corresponding to $f \in Lip_{0\mathcal{F}}(X,X)$ by Lemma \ref{Lem}(2) there exists $\hat{f} \in L(F(X),F(X))$ such that $\hat{f} \circ \delta_X = \delta_X \circ f$. Put $E:= span(f(X))$. Then $dim(E)< \infty$, as $f \in Lip_{0\mathcal{F}}(X,X)$.

    Let us now consider the inclusion map $\mathfrak{i} : E \xrightarrow{} X$. Thus there exists an isometry $\hat{\mathfrak{i}} \in L(F(E), F(X))$ such that $\hat{\mathfrak{i}} \circ \delta_E = \delta_X \circ \mathfrak{i}$. Since for any $x \in X$,
    $$\hat{f}(\delta_x)= \delta_X(f(x))=\delta_X(\mathfrak{i}(f(x)))= \hat{\mathfrak{i}}\delta_E(f(x)) \in \hat{\mathfrak{i}}(F(E)),$$ by linearity and continuity we have $\hat{f}(F(X)) \subset \hat{\mathfrak{i}}(F(E))$. Now $K$ is a $p$-compact subset of $X$. Therefore by \cite[Definition 2.1]{COWAFTSLP} $K$ is compact in $X$ and hence $\delta_E \circ f(K)$ is also compact.

    Since $E$ is finite dimensional, by \cite[Proposition 5.1]{LFBS} $F(E)$ has $1$-bounded approximation property. Therefore there exists $S \in \mathcal{F}((F(E)),(F(E)))$ with $\|S\| \leq 1$ such that 
    $$\sup\limits_{x \in K} \left\|S\delta_E(f(x)) -\delta_E(f(x))\right\|_{F(E)} \leq \frac{\epsilon}{2}.$$
    Now we consider the following composition operator, which will approximate $Id_{F(X)}$:
    $$F(X)\xrightarrow{\hat{f}} \hat{\mathfrak{i}}(F(E)) \xrightarrow{\hat{\mathfrak{i}}^{-1}} F(E) \xrightarrow{S} F(E) \xrightarrow{\hat{\mathfrak{i}}} F(X).$$
    Further 
    \begin{eqnarray*}
       & \sup\limits_{x \in K} \left\|\hat{\mathfrak{i}}S\hat{\mathfrak{i}}^{-1}\hat{f}\delta_X(x) - \delta_X(x)\right\|_{F(X)}&\\
        &\leq& \sup\limits_{x \in K} \left\|\hat{\mathfrak{i}}S\hat{\mathfrak{i}}^{-1}\hat{f}\delta_X(x) - \hat{\mathfrak{i}}\delta_E(f(x))\right\|_{F(X)}+ \sup\limits_{x \in K} \left\|\hat{\mathfrak{i}}\delta_E(f(x)) - \delta_X(x)\right\|_{F(X)}\\
        &\leq& \sup\limits_{x \in K} \left\|S\hat{\mathfrak{i}}^{-1}\hat{f}\delta_X(x) - \delta_E(f(x))\right\|_{F(E)}+ \sup\limits_{x \in K} \left\|\delta_X(f(x)) - \delta_X(x)\right\|_{F(X)}\\
        &=& \sup\limits_{x \in K} \left\|S\delta_E(f(x)) - \delta_E(f(x))\right\|_{F(E)}+ \sup\limits_{x \in K} \left\|\delta_X(f(x)) - \delta_X(x)\right\|_{F(X)}\\
        &=& \sup\limits_{x \in K} \left\|S\delta_E(f(x)) - \delta_E(f(x))\right\|_{F(E)}+ \sup\limits_{x \in K} \left\|f(x) - x\right\|_{F(X)}\\
        &\leq& \epsilon \ ,
    \end{eqnarray*}
    shows that $Id_{F(X)} \in \overline{\mathcal{F}(F(X),F(X))}^{\mathcal{T}\delta_{p}}$. 

    Conversely, let $Id_{F(X)} \in \overline{\mathcal{F}(F(X),F(X))}^{\mathcal{T}\delta_{p}}$, $K$ be a $p$-compact subset of $X$ and $\epsilon >0$. 
    Since $\left\{\beta_X S \delta_X : S \in \mathcal{F}(F(X),F(X))\right\} \subset Lip_{0\mathcal{F}}(X, X)$, $$\overline{\left\{\beta_X S \delta_X : S \in \mathcal{F}(F(X),F(X))\right\}}^{\tau_p} \subset \overline{Lip_{0\mathcal{F}}(X, X)}^{\tau_p}.$$
    We show that $Id_X \in \overline{\left\{\beta_X S \delta_X : S \in \mathcal{F}(F(X),F(X))\right\}}^{\tau_p}$.
    Now $Id_{F(X)} \in \overline{\mathcal{F}(F(X),F(X))}^{\mathcal{T}\delta_{p}}$ implies 
    $$\sup\limits_{x \in K}\left\|\beta_X \delta_X(x)-\beta_XS\delta_X(x)\right\|=\sup\limits_{x \in K}\left\|\beta_X (\delta_X(x)-S\delta_X(x))\right\| \leq \sup\limits_{x \in K}\left\|\delta_X(x)-S\delta_X(x)\right\| < \epsilon.$$
    Hence the proof follows.
\end{proof}
\begin{thm}
    A Banach space $X$ has Lipschitz $p$-approximation property if and only if $\beta_{X} \in \overline{\mathcal{F}(F(X),X)}^{\mathcal{T}\delta_{p}}$.
\end{thm}
\begin{proof}
      Let $K$ be a $p$-compact subset of a Banach space $X$ and $\epsilon>0$. Since $X$ has $Lip-p.a.p$, there exists $ f \in Lip_{0\mathcal{F}}(X,X)$ such that 
        $\sup\limits_{x \in K}\|Id_{X}(x) - f(x)\| < \epsilon$.  Therefore  $\sup\limits_{x \in K}\|Id_{X}(x) - f(x)\| = \sup\limits_{x \in K}\|\beta_{X}\delta_{X}(x) - T_{f}\delta_{X}(x)\| < \epsilon$, where $T_f$ be the linearization of $f$.
    Hence, the forward implication follows.
    
    Conversely, let $\beta_{X} \in \overline{\mathcal{F}(F(X),X)}^{\mathcal{T}\delta_{p}}$, $K$ be a $p-$compact subset of $X$ and $\epsilon>0$.
    Then there exists $T \in \mathcal{F}(F(X),X)$ such that $ \sup\limits_{x \in K}\|\beta_{X}\delta_{X}(x) - T\delta_{X}(x)\| < \epsilon$. Thus \  $\sup\limits_{x \in K}\|Id_{X}(x) - T\delta_{X}(x)\| < \epsilon$. Since $T\circ \delta_X \in Lip_{0\mathcal{F}}(X, X)$, $X$ has $Lip-p.a.p$.
\end{proof}
 The Lipschitz $p$-approximation property will be characterized in the paragraphs that follow. This characterization might be seen as an analog of a well-known result on the $p$-approximation property due to Karn and Sinha (see \cite[Proposition~6.2]{COWAFTSLP}).
   \begin{prop}
    The following are equivalent for a Banach space $X$:
    \begin{enumerate}
        \item $X$ has the Lipschitz $p$-approximation property.
        \item For all Banach spaces $Y$, $Lip_{0}(X,Y) = \overline{Lip_{0\mathcal{F}}(X,Y)}^{\tau_{p}}$.
    \end{enumerate}
\end{prop}
\begin{proof}
    Let $X$ have $Lip-p.a.p$, $Y$ be a Banach space, $f \in Lip_{0}(X,Y)$, $K$ be a $p$-compact subset of $X$ and $\epsilon>0$. \\
    Therefore by Theorem \ref{thm0.2} there exists $S \in \mathcal{F}(F(X),F(X))$ such that $\sup\limits_{x \in K}\|\delta_{X}(x) - S\delta_{X}(x)\| < \epsilon$.\\
    As $f \in Lip_{0}(X,Y)$ the linearization map $T_{f} \in L(F(X),Y)$ and $T_{f} \circ S \circ \delta_{X} \in Lip_{0\mathcal{F}}(X,Y)$.
    Now, 
    \begin{eqnarray*}
      \sup\limits_{x \in K}\|T_{f} \circ S \circ \delta_{X}(x) -f(x)\| &=&  \sup\limits_{x \in K}\|T_{f} \circ S \circ \delta_{X}(x) -T_{f}\circ \delta_{X}(x)\| \\
      &\leq& \epsilon\|T_{f}\|. 
    \end{eqnarray*}
    Hence it follows that $(1) \implies (2)$, and $(2) \implies (1)$ trivially follows by considering $Y=X$.
\end{proof}
\begin{rem}
    For all Banach spaces $Y$, $Lip_{0}(Y, X) = \overline{Lip_{0\mathcal{F}}(Y, X)}^{\tau_{p}}$ implies that $X$ has Lipschitz $p$-approximation property, but we don't know whether the converse part holds or not.
\end{rem}
We demonstrate another characterization of the Lipschitz $p$-approximation property of a Banach space. 
\begin{prop} 
    A Banach space $X$ has the Lipschitz $p$-approximation property if and only if 
     for any Banach space $Y$ and any $f \in Lip_{0}(X,F(Y))$,   $ \ \beta_{Y}\circ f \in \overline{Lip_{0\mathcal{F}}(X,Y)}^{\tau_{p}}$.  
\end{prop}
\begin{proof}
    Proof of the sufficient part directly follows by considering $Y=X$ and $f=\delta_{X}$.
    
    For the necessary part suppose $Y$ be a Banach space, $K$ be a $p$-compact subset of $X$, $\epsilon > 0$ and $f \in Lip_{0}(X,F(Y))$. Therefore by the previous Proposition there exists $g \in Lip_{0 \mathcal{F}}(X,F(Y))$ such that $\sup\limits_{x \in K}\|f(x) - g(x)\| < \epsilon $.
    
    Now $\sup\limits_{x \in K}\|\beta_{Y} \circ f (x) - \beta_{Y} \circ g (x)\| \leq \sup\limits_{x \in K}\|f(x) - g(x)\| < \epsilon$ and it immediately completes the proof.
\end{proof}
\begin{prop}
    For a Banach space $X$, if $F(X)$ has the Lipschitz approximation property, then $X$ has the Lipschitz $p$-approximation property.
\end{prop}
\begin{proof}
    Let $K$ be a $p$-compact subset of $X$ and $\epsilon>0$. Since $F(X)$ has Lipschitz approximation property and $\delta_X(K)$ is compact; there exists $S \in Lip_{0\mathcal{F}}(F(X),F(X))$ such that $\sup\limits_{x \in K}\|\delta_{X}(x) - S\delta_{X}(x)\| < \epsilon$.\\
    Now 
    \begin{eqnarray*}
         \sup\limits_{x \in K}\|Id_{X}(x) - \beta_{X}S\delta_{X}(x)\| &\leq& \sup\limits_{x \in K}\|\delta_{X}(x) - S\delta_{X}(x)\| < \epsilon
    \end{eqnarray*}
    Hence $Id_{X} \in \overline{Lip_{0\mathcal{F}}(X,X)}^{\tau_{p}}$ and $X$ has $Lip-p.a.p$.
\end{proof}
The following lemma is easy to prove, which will be used further.
\begin{lem}
    For any two Banach spaces $X, Y,$ and $T \in L(X, Y)$, $\mathcal{J}_{Y}(Tv) = T^{**}(\mathcal{J}_{X}(v))$ for all $v \in X$, where $\mathcal{J}$ is the canonical embedding map from a normed linear space to its double dual space. 
\end{lem}
\begin{lem}\label{lem4.13}
    For any Banach spaces $X,Y$; $\ L(X,Y) \subset \overline{\mathcal{F}(X,Y)}^{\tau_{p}} $ if and only if $\ T\beta_{X} \in \overline{\mathcal{F}(F(X),Y)}^{\tau_{p}}$ for all $ T \in L(X,Y)$.
\end{lem}
\begin{proof}
    Suppose $T \in L(X,Y)$, $K$ is a $p$-compact subset of $F(X)$, and $\epsilon>0$. Since $\beta_X(K)$ is $p$-compact in $X$, there exists $S \in \mathcal{F}(X,Y)$ such that $\sup\limits_{\gamma \in K}\|T\beta_{X}(\gamma) - S\beta_{X}(\gamma)\| < \epsilon$, and hence the forward part follows immediately.
    
    Conversely, let $T \in L(X,Y)$ such that $T\beta_{X} \in \overline{\mathcal{F}(F(X),Y)}^{\tau_{p}}$, $K$ be a $p$-compact subset of $X$ and $\epsilon>0$. \\
    Consider $E = \overline{span(K)}$, a separable closed subspace of $X$, as being a $p$-compact set $K$ is compact. Therefore, by \cite[Lemma~3.4]{APITLM} there exists $ U : X \xrightarrow{} F(X)^{\ast \ast}$ bounded linear map such that $\|U\| = 1, \ \beta_{X}^{\ast \ast}Uv= \mathcal{J}_{X}v \ \ \forall\ v \in X$, and $U(E) \subset F(X).$  \\
    Since $T\beta_{X} \in \overline{\mathcal{F}(F(X),Y)}^{\tau_{p}}$ and $U(K)$ be a $p$-compact subset of $F(X)$; there exists $S \in \mathcal{F}(F(X),Y)$ such that $\sup\limits_{v \in K}\|T\beta_{X}U(v) - SU(v)\| < \epsilon$.
    Now, for any $v \in K$
    \begin{eqnarray*}
        \|SU(v)-T(v)\| &=& \|\mathcal{J}_{Y}(SU(v)) -\mathcal{J}_{Y}(T(v)) \|\\
        &=& \|S^{\ast \ast}\mathcal{J}_{F(X)}(U(v)) -\mathcal{J}_{Y}(T(v)) \|\\
        &=& \|S^{\ast \ast}\mathcal{J}_{F(X)}(U(v)) -T^{**}\beta_{X}^{\ast \ast}Uv \|\\
        &=& \|\mathcal{J}_{Y}(SU(v)) - (T\beta_{X})^{\ast \ast}(\mathcal{J}_{F(X)}(Uv))\|\\
        &=& \|\mathcal{J}_{Y}SUv - \mathcal{J}_{Y}T\beta_{X}Uv\|\\
        &=&\|T\beta_{X}U(v) - SU(v)\| \\
        &<& \epsilon.
    \end{eqnarray*}
    which gives $T \in \overline{\mathcal{F}(X,Y)}^{\tau_{p}}$.
\end{proof}
The previous lemma gives us the following result.
\begin{cor}
     A Banach space $X$ has the $p$-approximation property if and only if $T\beta_{X} \in \overline{\mathcal{F}(F(X), Y)}^{\tau_{p}} \ \mbox{for all} \ T \in L(X, Y)$, and for any Banach space $Y$.
\end{cor}
The proof of the next proposition directly follows from the above lemma \ref{lem4.13}.
\begin{prop}
    For any Banach space $X$, if $F(X)$ has the $p$-approximation property, then $X$ has $p$-approximation property, and further $X$ has the Lipschitz $p$-approximation property.
\end{prop}
\begin{proof}
    Let $Y$ be a Banach space and $T \in L(X, Y)$. Since $F(X)$ has $p.a.p$, $L(F(X), Y) = \overline{\mathcal{F}(F(X), Y)}^{\tau_p}$. Now $T\beta_X \in L(F(X), Y) = \overline{\mathcal{F}(F(X), Y)}^{\tau_p}$, and hence $X$ has $p.a.p$.
\end{proof}
\subsection{Some examples}
From the previous discussion, we can immediately deduce the following facts about the Banach space $X.$ If $X$ has the approximation property then $X$ has $Lip-a.p,$ $Lip-p.a.p,$ $p.a.p$. Also, we know that any Banach space with a Schauder basis has the approximation property. In \cite{OSBILFS} H\'ajek and Perneck\'a proved that $F(\ell_{1})$ and $F(\mathbb{R}^{n})$ have a Schauder basis. Therefore, these are examples of Lipschitz-free spaces of $\ell_{1}$ and $\mathbb{R}^{n}$, respectively, having $Lip-a.p,$ $Lip-p.a.p,$ $p.a.p$.\\
 \cite[Theorem~6.4]{COWAFTSLP} provides that any Banach space $X$ has $p$-approximation property (for $1\leq p \leq 2$), and hence $X$ has $Lip-p.a.p$ (for $1\leq p \leq 2$).\\
 From the above section and \cite{APITLM}, we have the following diagram as a summary.
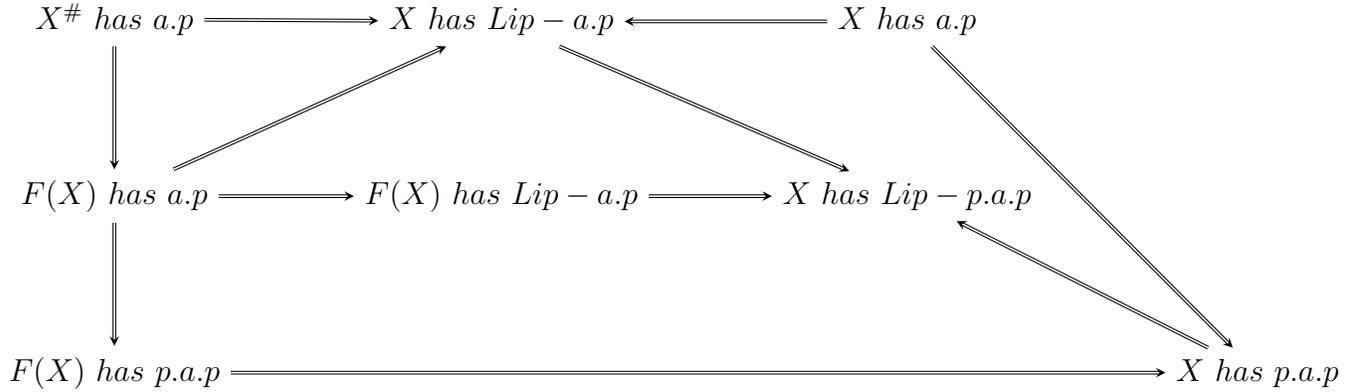
\begin{center}
\begin{tikzpicture}
\matrix (m) 
[matrix of math nodes,row sep=4em,column sep=4em,minimum width=2em] 
{
X^{\#}\ has \ a.p & X \ has \ Lip-a.p & X \ has \ a.p \\
           F(X) \ has \ a.p & F(X) \ has \ Lip-a.p & X \ has \ Lip-p.a.p \\ 
           F(X)\ has \ p.a.p & & & X\ has\ p.a.p\\
}; 
\path[-stealth] 
(m-1-1) edge [double] node [above] {} (m-1-2)
        edge [double] node [below] {} (m-2-1)
(m-2-1) edge [double] node [right] {} (m-1-2)
(m-2-1) edge [double] node [right] {} (m-2-2)
(m-1-3) edge [double] node [left] {} (m-1-2)
(m-1-3) edge [double] node [right] {} (m-3-4)
(m-2-2) edge [double] node [right] {} (m-2-3)
(m-1-2) edge [double] node [right] {} (m-2-3)
(m-2-1) edge [double] node [below] {} (m-3-1)
(m-3-1) edge [double] node [right] {} (m-3-4)
(m-3-4) edge [double] node [left] {} (m-2-3);
\end{tikzpicture}
\end{center}

\section{\bf Factorization of Lipschitz operators that belong to the dual operator ideal of Lipschiz $p$ -compact operator} \label{FOLPCO}
In \cite{COWFTSLP} Karn and Sinha showed that $\mathcal{K}^d_p$ may be regarded as a
compact-type operator ideal whose associated maximal operator ideal is
$\Pi_p$. More precisely, they proved the composition/decomposition relation
$$
\Pi_p \circ \mathcal{K} = \mathcal{K}^d_p \qquad (1 \leq p \leq \infty).
$$
Here in this section, we establish an analogous decomposition relation in the Lipschitz setting. Recall that any compact operator between Banach spaces can be expressed as the composition of a bounded linear operator with a compact operator (see \cite[pg 225]{TVSII}). We have obtained a corresponding result here for Lipschitz compact operators.

\begin{thm} \label{THM2.1}
Let $f \in Lip_{0}(X, Y)$. Then $f \in \mathcal{K}_{p}^{Ld}(X,Y)$ if and only if $T_{f} \in \mathcal{K}_{p}^{d}(F(X),Y)$ with $k_{p}^{d}(T_{f}) = k_{p}^{Ld}(f).$
\end{thm}
\begin{proof}
Let $f \in \mathcal{K}_{p}^{Ld}(X,Y)$. Therefore, by \cite[Definition 3.8]{LOITAP}, $f^{t} \in \mathcal{K}_{p}(Y^{*}, X^{\#})$. Now, $T_{f}^{*} = Q_{X} \circ f^{t} \in \mathcal{K}_{p}(Y^{*}, F(X)^{*})$ (by \cite[Theorem~3.1]{LCO} and ideal \ property). Hence $T_{f} \in \mathcal{K}_{p}^{d}(F(X),Y)$. 

Conversely, let $f \in Lip_{0}(X,Y)$ such that  $T_{f} \in \mathcal{K}_{p}^{d}(F(X),Y)$. Thus by \cite[Proposition~3.2]{LPCM} $T_{f}^{*} \in \mathcal{K}_{p}(Y^{*}, F(X)^{*}) \subset \mathcal{K}_{p}^{L}(Y^{*}, F(X)^{*})$ and using the ideal property we have $f^{t} = Q_{X}^{-1} \circ T_{f}^{*} \in \mathcal{K}_{p}^{L}(Y^{*}, X^{\#})$. Therefore, it immediately follows that $f \in \mathcal{K}_{p}^{Ld}(X,Y)$.\\
Equality of the norms part has been proved in Lemma \ref{Lem 2.6}.
\end{proof}
\begin{prop}
Let $X$ be a normed linear space and $Y$ be a Banach space. Then $\mathcal{K}_{p}^{Ld}(X,Y) \subset \prod_{p}^{L}(X,Y)$ with $\pi_{p}^{L}(f) \leq k_{p}^{Ld}(f).$
\end{prop}
\begin{proof}
    Let $f \in \mathcal{K}_{p}^{Ld}(X,Y)$, then by Theorem \ref{THM2.1} $T_{f} \in \mathcal{K}_{p}^{d}(F(X),Y) \subset \prod_{p}(F(X),Y)$ (by \cite[Proposition~5.3]{COWAFTSLP}). Since $T_{f}$ is linear, using \cite[Theorem 2]{LPSO}, we have $T_{f} \in \prod_{p}^{L}(F(X),Y)$. Now, $f = T_{f} \circ \delta_{X} \in \prod_{p}^{L}(X,Y)$ (by \cite[Proposition 2.5]{LOITAP}).
    
    For the norm inequality using \cite[Theorem~2]{LPSO}, \cite[Remark~2(c)]{COWAFTSLP}, and Lemma \ref{Lem 2.6}(2)), respectively, we get $\pi_{p}^{L}(f) = \pi_{p}^{L}(T_{f} \circ \delta_{X}) \leq \pi_{p}^{L}(T_{f})= \pi_{p}(T_{f})\leq k_{p}^{d}(T_{f})= k_{p}^{Ld}(f)$.
\end{proof}
\begin{prop}
    Let $X$ and $Y$ be two Banach spaces. Then $\mathcal{K}_{p}^{d}(X,Y) \subset \mathcal{K}_{p}^{Ld}(X,Y).$ 
\end{prop}
\begin{proof}
    Suppose $f \in \mathcal{K}_{p}^{d}(X,Y)$, then by \cite[Proposition~3.2]{LPCM}$, f^{t} = f^{*} \in \mathcal{K}_{p}^{L}(Y^{*},X^{\#})$. Hence, it follows that $f \in \mathcal{K}_{p}^{Ld}(X,Y)$.
\end{proof}
\begin{rem} If $f \in \mathcal{K}_{p}^{Ld}(X,Y)$ is a linear map, then $f \in \mathcal{K}_{p}^{d}(X,Y)$, and in this case $k_{p}^{Ld}(f)=k_{p}^{d}(f)$.
\end{rem}
\begin{lem} \label{Lem2.5}
    Let  $X,Y$, and $G$ be Banach spaces, and $\mathfrak{I} \in L(Y,G)$ be an isometric isomorphism, and $T \in \mathcal{K}_{p}(X,Y)$, then $k_{p}(\mathfrak{I}T) = k_{p}(T)$.
\end{lem}
\begin{proof}
    As $T \in \mathcal{K}_{p}(X,Y)$, there exists $y=(y_{n}) \in \ell_{p}^{s}(Y)$ such that $$T(B_{X}) \subset E_{y}(B_{\ell_{p^{'}}}):= \left\{\sum\limits_{n=1}^{\infty} a_n y_n : \ (a_n) \in B_{\ell_{p^{'}}}\right\},$$ which is equivalent to $\mathfrak{I}T(B_{X}) \subset E_{(\mathfrak{I}y_{n})}(B_{\ell_{p^{'}}})$. Hence, the proof follows.
\end{proof}
One can easily verify the first part of the following lemma: 
\begin{lem} \label{Lem 2.6}
Let $X$ be a normed linear space, and $Y$ and $Z$ be Banach spaces; $f \in Lip_{0}(X, Y),\ \mathfrak{I} \in L(Y, Z)$ be an isometry. Then 
\begin{enumerate}
    \item $T_{\mathfrak{I}f} = \mathfrak{I}T_{f}$
    \item $k_{p}^{d}(T_{f}) = k_{p}^{Ld}(f)$
    \begin{proof}
     By definition, we have $ k_{p}^{d}(T_{f}) = k_{p}(T_{f}^{*})$. It follows $k_{p}(T_{f}^{*})=k_{p}(Q_{X} \circ f^{t})=k_{p}(f^{t})$, from \cite[Theorem~3.1]{LCO} and by the previous Lemma \ref{Lem2.5}. Now, \cite[Proposition~3.2]{LPCM} implies that $k_{p}(f^{t})=k_{p}^{L}(f^{t})=k_{p}^{Ld}(f)$.
    \end{proof}
    \item $f \in \mathcal{K}_{p}^{Ld}(X,Y)$ if and only if $\mathfrak{I}f \in \mathcal{K}_{p}^{Ld}(X,Z)$ with $k_{p}^{Ld}(\mathfrak{I}f) = k_{p}^{Ld}(f)$.
    \begin{proof}
        Suppose, $\mathfrak{I}f \in \mathcal{K}_{p}^{Ld}(X,Z)$. Therefore, Theorem \ref{THM2.1} gives us $T_{\mathfrak{I}f} \in \mathcal{K}_{p}^{d}(F(X),Z)$. Then, in the first part of this lemma, we get $\mathfrak{I}T_{f} \in \mathcal{K}_{p}^{d}(F(X), Z)$, which is equivalent as $T_{f} \in \mathcal{K}_{p}^{Ld}(F(X), Y)$,(by \cite[Lemma 2.5]{COWFTSLP}). Now Theorem~\ref{THM2.1} provides us $T_{f} \in \mathcal{K}_{p}^{d}(F(X), Y)$ if and only if $f \in \mathcal{K}_{p}^{Ld}(Y,Z)$.
        
    Further, $k_{p}^{Ld}(\mathfrak{I}f)= k_{p}^{d}(T_{\mathfrak{I}f})= k_{p}^{d}(\mathfrak{I}T_{f}) = k_{p}^{d}(T_{f})= k_{p}^{Ld}(f)$ (using part $(2),(1)$ of this lemma and  \cite[Lemma~2.5]{COWFTSLP}).
    \end{proof}
\end{enumerate}
\end{lem}
\begin{lem} \label{lem2.7}
Let $X$ be a normed linear space and $Y$ be Banach space, then for $f \in Lip_{0}(X,Y)$ with $ T_{f} \in \prod_{p}(F(X),Y)$ implies $f \in \prod_{p}^{L}(X,Y)$ with $\pi_{p}^{L}(f) \leq \pi_{p}(T_{f})$.
\end{lem}
\begin{proof}
The proof of the first part trivially follows from the fact that $f = T_f \circ \delta_X$.

    Now $f = T_{f} \circ \delta_{X} \implies \pi_{p}^{L}(f)=\pi_{p}^{L}(T_{f} \circ \delta_{X})\leq \pi_{p}^{L}(T_{f}) = \pi_{p}(T_{f})$ (by ideal property and \cite[Theorem 2]{LPSO}).
    \end{proof}
\begin{rem}
$\bullet$ The other-way implication of the above lemma is not true. As a counterexample, for $p=1$ consider $\delta_{\mathbb{R}} \in \prod_{p}^{L}(\mathbb{R},F(\mathbb{R}))$, but $T_{\delta_{\mathbb{R}}} = Id_{F(\mathbb{R})} \notin \prod_{p}(F(\mathbb{R}),F(\mathbb{R}))$.\\
$\bullet$ For a normed linear space $X$, and a finite-dimensional vector space $E$, if $f \in \mathcal{K}_{p}^{Ld}(X,E)$ then $k_{p}^{Ld}(f)=k_{p}^{d}(T_{f}) = \pi_{p}(T_{f})$ (using Lemma \ref{Lem 2.6}(2), \cite[Theorem 2.6]{COWFTSLP}), for $1 \leq p < \infty$.
\end{rem}
\subsection{Factorization of $\mathcal{K}_{p}^{Ld}$}
\begin{prop}\label{3.10}
    Let $X$ be a normed linear space, $Y$ be a Banach space, $1 \leq p < \infty$, then $f \in \mathcal{K}_{p}^{Ld}(X,Y)$ if and only if there exists a Banach space $Z$, $R \in Lip_{0\mathcal{K}}(X,Z)$, $U \in \prod_{p}(Z,Y)$ such that $f = UR$.
\end{prop}
\begin{proof}
    Suppose $f \in \mathcal{K}_{p}^{Ld}(X,Y)$. Then by Theorem \ref{THM2.1}, $ T_{f} \in \mathcal{K}_{p}^{d}(F(X),Y)$. Again, by \cite[Theorem~3.1]{COWFTSLP} there exist a Banach space $Z$, $V \in \mathcal{K}(F(X),Z)$ and $U \in \prod_{p}(Z,E)$ such that $T_{f} = UV \ \mbox{that is.} \ f = T_{f} \circ \delta_{X} = UV\delta_{X} = UR$ where $R = V\delta_{X} \in Lip_{0\mathcal{K}}(X,Z)$ and $U \in \prod_{p}(Z,Y) \subset \prod_{p}^{L}(Z,Y)$.
    
    Conversely, let $Z$ be a Banach space, $f \in Lip_{0}(X,Y)$, and there exists $R \in Lip_{0\mathcal{K}}(X,Z)$, $U \in \prod_{p}(Z,Y)$ such that $f = UR$. Now it is easy to verify that $T_{f} = UT_{R} \in \mathcal{K}_{p}^{d}(F(X),Y)$, which implies that $f \in \mathcal{K}_{p}^{Ld}(X,Y)$.
\end{proof}
\begin{prop}\label{3.11}
    Let $X$ and $Y$ be Banach spaces, and $f \in Lip_{0\mathcal{K}}(X,Y)$. Then $f$ can be decomposed as $f= Cg$, where $C \in Lip_{0}(g(X),Y)$ and $g \in Lip_{0\mathcal{K}}(X,c_0)$, where $g(X)$ is considered as a pointed metric subspace of $c_0$.
\end{prop}
\begin{proof}
     Suppose $X$ and $Y$ are Banach spaces, and $f \in Lip_{0K}(X, Y)$ is a Lipschitz compact operator. Therefore, by \cite[Proposition 3.5]{LCO}, $f^{t} \in \mathcal{K}(Y^{*},X^{\#})$ and hence by using \cite[Proposition 1.e.2]{CBS} we get $f^{t}(B_{Y^{*}}) \subset E_{(u_{n})}(B_{\ell_{1}}):= \left\{\sum\limits_{n=1}^{\infty} a_n u_n : \ (a_n) \in B_{\ell_1}\right\}$ for some $(u_{n}) \in c_{0}(X^{\#})$. Consider the map $ g: X \xrightarrow{} c_{0}$, defined as $g(x)=(u_{n}(x))$ for all $x \in X$. Since $\left\{\frac{g(x_1)-g(x_2)}{\|x_1-x_2\|}: x_1 \neq x_2; \ x_1,x_2 \in X\right\} \subset co\{(\|u_{n}\|): n \in \mathbb{N}\}$, we have $g \in Lip_{0\mathcal{K}}(X, c_0)$. 

     Let $x_1, x_2 \in X$ be such that $u_n(x_1) = u_n(x_2)$ for all $n \in \mathbb{N}$. We show that $f(x_1) = f(x_2)$. Let $y^* \in B_{Y^*}$. Then there exists a sequence $(\alpha_n) \in B_{\ell_1}$ such that $(f^t)(y^*) = \sum\limits_{n=1}^{\infty} \alpha_n u_n$. Thus 
    \begin{eqnarray*}
    	y^*(f(x_1)) &=& (f^t)(y^*)(x_1) \\ 
    	&=& \sum\limits_{n=1}^{\infty}\alpha_n u_n(x_1) \\ 
    	&=& \sum\limits_{n=1}^{\infty} \alpha_n u_n(x_2) \\ 
   	&=& (f^t)(y^*)(x_2) \\ 
   	&=& y^*(f(x_2)).
    \end{eqnarray*}
    Since $y^* \in B_{Y^*}$ is arbitrary, we get $f(x_1) = f(x_2)$. Therefore, the map $C: g(X) \xrightarrow{} Y$ given by $C(g(x))=f(x)$ is well-defined. We show that $C$ is Lipschitz. Fix $x_1, x_2 \in X$. If $y^* \in B_{Y^{*}}$, then as above $y^*(f(x)) = (f^t)(y^{\ast})(x) = \sum\limits_{n=1}^{\infty} \alpha_n u_n(x)$. Thus 
    $$\vert y^*(f(x_1)-f(x_2)) \vert = \vert \sum\limits_{n=1}^{\infty} \alpha_n (u_n(x_1)-u_n(x_2)) \vert \le \sup \lbrace \Vert u_n(x_1)-u_n(x_2) \Vert: n \in \mathbb{N} \rbrace = \Vert g(x_1)-g(x_2) \Vert.$$ 
    Hence 
    $$\Vert C(g(x_1)) -C(g(x_2))\Vert = \Vert f(x_1) -f(x_2)\Vert = \sup_{y^* \in B_{Y^*}} \vert y^*(f(x_1)-f(x_2)) \vert \le \Vert g(x_1) -g(x_2)\Vert.$$
    so that $C$ is Lipschitz. Hence, the proposition follows.
\end{proof}
Playing with the Lipschitz map and its corresponding linearization map leads to the proof of the following proposition.
\begin{prop}\label{prop3.14}
    Let $X$ and $Y$ be Banach spaces. Then  $\overline{Lip_{0\mathcal{F}}(Y, X)}^{k_{p}^{Ld}} = \mathcal{K}_{p}^{Ld}(Y, X)$ if and only if $\overline{\mathcal{F}(F(Y), X)}^{k_{p}^{d}}=\mathcal{K}_{p}^{d}(F(Y), X)$. 
\end{prop}
The definition of a Banach space having an approximation property of type $p$ is used in the next proposition and can be found in \cite[Definition~4.4]{COWFTSLP}.
\begin{prop}
    If a Banach space $X$ has the approximation property of type $p$, then $\overline{Lip_{0\mathcal{F}}(Y,X)}^{k_{p}^{Ld}} = \mathcal{K}_{p}^{Ld}(Y,X)$. 
\end{prop}
\begin{proof}
    The proof follows directly from the definition and the previous proposition \ref{prop3.14}.
\end{proof}
\begin{prop}
    Let $X$ be a Banach space having the $p$-approximation property. Then for all Banach space $Y$; $\mathcal{K}_{p}^{L}(Y,X) \subset \overline{Lip_{0\mathcal{F}}(Y,X)}^{Lip(.)}$.
\end{prop}
\begin{proof}
    Let $Y$ be a Banach space, $f \in \mathcal{K}_{p}^{L}(Y,X)$ and $\epsilon>0$. Therefore, as $X$ has $p.a.p$ and $Im_{Lip}(f) :=\left\{\frac{f(x_1)-f(x_2)}{\|x_1-x_2\|} : x_1 \neq x_2\right\}$ is relatively $p$-compact, there exists $S \in \mathcal{F}(X,X)$ such that $\sup\limits_{z \in Im_{Lip}(f)} \|Sz -z\|<\epsilon$. Hence, it gives $Lip(Sf-f) <\epsilon$, which completes the proof.
\end{proof}
In the following proposition, depending on the properties of the Lipschitz map $\delta_{X}$, we characterize their domain.
\begin{prop}
    For any Banach space $X$, $\beta_{X}: F(X)\xrightarrow{}X$ is compact iff $dim(X)$ is finite. 
\end{prop}
\begin{proof}
    Observe that $\delta_{X}(B_{X}) \subset B_{F(X)}$ and this will give us $\beta_{X}$ is compact implies $dim(X)<\infty$.
    The converse part is easy to show. 
\end{proof}
\begin{prop}
    For a Banach space $X$, if $\delta_{X}:X\xrightarrow{}F(X) $ is Lipschitz $p$-compact, then $dim(X)$ is finite.
\end{prop}
\begin{proof}
    Let $Y$ be any Banach space and $f \in Lip_{0}(X,Y)$. Now $\beta_{Y} \hat{f}\delta_{X} =f$ and $Im_{Lip}(f)= \beta_{Y}\hat{f}(Im_{Lip}(\delta_{X}))$. This implies that $f$ is Lipschitz $p$-compact and hence Lipschitz compact. Therefore, in particular, for $Y=X$ and $f =Id_{X}$, $f$ is $p$-compact. Hence, this is possible only if  $X$ is of finite dimension.
\end{proof}
\begin{rem}
    The converse of the above proposition is not true. Consider $\delta_{\mathbb{R}}$ (is never a Lipschitz $p$-compact) as a counter-example. 
\end{rem}
\begin{itemize}
    \item We observed the facts that $X^{\#}$ or $F(X)$ has approximation property implies $X$ has $Lip-a.p$ and also that $F(X)$ has $Lip-a.p$ implies $X$ has $Lip-p.a.p$. 
These would be interesting questions to analyze the converse part of those facts.
\item Examples of the Banach spaces without having $Lip-a.p, \ Lip-p.a.p.$ 
\item  We have seen that $a.p \implies Lip-a.p \implies Lip-p.a.p;$ also, $a.p \implies p.a.p \implies Lip-p.a.p.$ So the immediate question to ask is, What can we conclude about the other-way implications for all the cases?
\end{itemize}
 We leave these questions open for the time being.
 
 {\bf Acknowledgement.} The author gratefully acknowledges Professor Anil Karn for his valuable guidance in understanding linear $p$-compact sets and $p$-approximation property.
 
{\bf Declaration.} 
The author acknowledges that he was financially supported by the Senior Research Fellowship from the National Institute of Science Education and Research,s Bhubaneswar, funded by the Department of Atomic Energy, Government of India. The authors have no competing interests or conflicts of interest to declare that are relevant to the content of this article.


\end{document}